\documentclass{amsart}
\usepackage{amssymb}
\usepackage{graphicx}
\usepackage[T1]{fontenc}
\usepackage[colorlinks,bookmarks]{hyperref}
\usepackage{tikz}
\usepackage{tikz-cd}
\usetikzlibrary{matrix,arrows}
\usetikzlibrary{arrows,positioning,shapes}
\usepackage[section]{placeins}
\usepackage[mathscr]{euscript}
\usepackage{enumitem}
\usepackage[all,cmtip]{xy}
\usepackage{float}

\newtheorem{lemma}{Lemma}[subsection]
\newtheorem{theorem}[lemma]{Theorem}
\newtheorem{proposition}[lemma]{Proposition}
\newtheorem{corollary}[lemma]{Corollary}

\newtheorem{conjecture}[lemma]{Conjecture}
\newtheorem{problem}[lemma]{Problem}

\theoremstyle{definition}

\newtheorem{remark}[lemma]{Remark}

\newcommand{\Z}{\mathbb{Z}}			
\newcommand{\R}{\mathbb{R}}			
\newcommand{\into}{\hookrightarrow}

\title{Deletion and contraction in configuration spaces of graphs}
\author{Sanjana Agarwal, Maya Banks, Nir Gadish$^1$, Dane Miyata}
\thanks{$^1$ N.G. is supported by NSF Grant No. DMS-1902762}


\begin{document}

\maketitle

\begin{abstract}
    The aim of this article is to provide space level maps between configuration spaces of graphs that are predicted by algebraic manipulations of cellular chains. More explicitly, we consider edge contraction and half-edge deletion, and identify the homotopy cofibers in terms of configuration spaces of simpler graphs. The construction's main benefit lies in making the operations functorial - in particular, graph minors give rise to compatible maps at the level of fundamental groups as well as generalized (co)homology theories.
    
    As applications we provide a long exact sequence for half-edge deletion in any generalized cohomology theory, compatible with cohomology operations such as the Steenrod and Adams operations, allowing for inductive calculations in this general context. We also show that the generalized homology of unordered configuration spaces is finitely generated as a representation of the opposite graph minor category. 
\end{abstract}

\section{Introduction}

For every graph $\Gamma$ (a finite 1-dimensional $CW$ complex) denote the configuration space of $n$ distinct points on $\Gamma$ by
\[
Conf_n(\Gamma) := \{ (x_1,\ldots,x_n)\in \Gamma^n \mid \forall i<j \; x_i\neq x_j \}.
\]
The symmetric group $S_n$ acts on this space by permuting the labels, and the quotient by this action is the unordered configuration space $UConf_n(\Gamma)$.

This paper provides space level maps between such configuration spaces corresponding to deletion and contraction of edges in the graph. Such constructions have been predicted by algebraic manipulations of cellular chains and applied in finite generation proofs. The immediate implication of our constructions is that various deletion and contraction operations on graphs induce well-defined maps on the fundamental groups of the configuration spaces as well as on any generalized (co)homology theory with its cohomology operations such as Steenrod and Adams operations.

Specifically, we discuss the following two constructions:
\begin{itemize}
    \item It has been observed (\cite[Lemma C.7]{an2017subdivisional}) that edge contraction on the graph induces well defined maps on a chain model of the graph's configuration spaces. This chain contraction was implicitly known to lift to a unique homotopy class of space level maps, as explained in \cite[Remark 1.10]{miyata2020categorical}. We construct a (zigzag of) space level maps exhibiting this homotopy class.
    
    In particular, this allows us to describe the homotopy cofiber
    of the contraction map as a certain configuration space with constraints.
    
    \item We identify the homotopy cofiber of the inclusion of a subgraph with some half-edges deleted with (suspensions of) the configuration space of a simpler graph.
    This gives an inductive tool for computing topological invariants via the long exact sequence of a pair, with all three terms being configuration spaces of graphs.
\end{itemize}

\begin{remark}[Functoriality]
All of our constructions will be obviously functorial in the data of a graph $\Gamma$ along with some additional input such as a choice of vertex, a set of half-edges, or a topological sub-tree. We will not belabor the point of this functoriality throughout the constructions.
\end{remark}

Let us discuss a number of applications of our constructions.

\subsection{Application 1: LES in K-theory and generalized cohomology}
Using the characterization of homotopy cofiber for half-edge deletion in \S\ref{sec:deletion} one gets the following inductive machinery for computing generalized homology and cohomology of configuration spaces. This sequence is already new for ordinary homology when deleting only one half-edge.
\begin{theorem}\label{thm:LES}
Let $\Gamma$ be a graph and fix a vertex $v\in V(\Gamma)$ along with a set of half edges $H=\{h_1,\ldots, h_r\}$ incident on $v$. For every generalized cohomology theory $E$ one has a long exact sequence, compatible with cohomology operations and natural with respect to graph embeddings and automorphisms
\begin{eqnarray*}
\ldots \to E^i(UConf_n(\Gamma \setminus \cup H)) \overset{\oplus\left(  add_v^*-add_h^* \right)}{\longrightarrow} E^{i}(UConf_{n-1}(\Gamma \setminus \{v\}))^{\oplus |H|} \overset{d}{\longrightarrow} \qquad\qquad \\
\longrightarrow E^{i+1}(UConf_n(\Gamma))      \overset{\iota_H^*}{\longrightarrow} E^{i+1}(UConf_n(\Gamma \setminus \cup H)) \to \ldots
\end{eqnarray*}

This LES is a reflection of the Puppe cofiber sequence of the deletion given in \S\ref{sec:applications}. In particular a similar LES exists for $E$-homology, as well as the $S_n$-equivariant versions for ordered configuration space.
\end{theorem}
\begin{remark}
The above LES in the special case in which $E$ is ordinary homology and with $H$ the set of \emph{all} half edges incident on $v$ has been a central tool in the work of An--Drummond-Cole--Knudsen \cite{an2020edge}, where it was discovered and used as a chain-level algebraic manipulation.

One of the advantages of our version above is that one can elect to remove one half-edge at a time, thereby always considering a single configuration space in every term.
\end{remark}

Another interesting special case is the LES in K-theory, respecting Adams operations, with all three terms being configuration spaces of graphs -- a result that could be of interest in quantum physics in light of \cite{macikazek2019non}.

\subsection{Application 2: contraction maps on graph braid groups}
As mentioned in \cite[Remark 1.10]{miyata2020categorical}, the edge contraction maps between configuration spaces are associated with well-defined homomorphisms between the respective fundamental groups -- the so called \emph{graph braid groups}. Our explicit space level construction of these maps provides a way to study the induced maps on $\pi_1$ directly.

\begin{problem}
Describe the edge contraction homomorphism between graph braid groups, e.g. with respect to the Farley-Sabalka presentation of tree braid groups \cite{farley2005discrete}.
\end{problem}

For our next couple of items we observe that the space level lift of edge contraction turns the assignments
\[
\Gamma\; \longmapsto \; UConf_n(\Gamma) \quad \text{ and } \quad Conf_n(\Gamma)
\]
into functors from the Miyata-Ramos-Proudfoot opposite graph-minor category \cite{miyata2020categorical} to the homotopy category of spaces.
In particular, the application of any homotopy invariant functor, such as $\pi_1$ and generalized homology theories gives representations of this category.

We consider fundamental groups first. The action of graph minors on $\Gamma\mapsto \pi_1(UConf_n(\Gamma))$ allows us to make the following conjecture:
Let $\gamma_i \pi(-)$ by the terms in the lower central series of $\pi_1(UConf_n(-))$, considered as representations of the opposite graph-minor category.
\begin{conjecture}[\textbf{Finite generation of LCS quotients}]
Every successive quotient of the LCS, $\gamma_i \pi(-)/\gamma_{i+1} \pi(-)$, forms a finitely generated representation of the opposite graph-minor category.
\end{conjecture}
\begin{remark}
The above conjecture is known to hold in the case $i=1$, equivalently for $H_1(UConf_n(-))$ (see \cite{miyata2020categorical}). Then a possible path to proving the conjecture would be to show that the Lie ring $\oplus_{i=1}^{\infty} \gamma_i \pi/\gamma_{i+1} \pi$ is generated by its degree $1$ elements.
\end{remark}

\subsection{Application 3: finite generation for generalized homology theories}
As above, consider $UConf_n(-)$ as a functor from the Miyata-Ramos-Proudfoot opposite graph-minor category \cite{miyata2020categorical} to the homotopy category of spaces. Then the application of a generalized homology theory gives a linear representation of this category. For these representations we prove,
\begin{theorem} \label{thm:fin-gen-homology}
Let $E$ be any connective multiplicative generalized homology theory such that its coefficient ring $E_*$ is Noetherian. Then for every $i\in \Z$, the functor
\[
\Gamma \mapsto E_i(UConf_n(\Gamma))
\]
is a finitely generated representation of the opposite graph-minor category.

Explicitly, this implies that for every fixed $n$ and $i$ there exist finitely many graphs $\Gamma_1,\ldots,\Gamma_r$ and $E$-homology classes $\alpha_j\in E_i(UConf_n(\Gamma_j))$ whose images under deletion and contraction of graphs span $E_i(UConf_n(\Gamma))$ for every graph $\Gamma$. 
\end{theorem}

\begin{remark}[\textbf{Generalized cohomology theories}]
The theory of graph-anyons in \cite{macikazek2019non} expresses the interest of quantum physicists in vector bundles over $UConf_n(G)$, and thus in the $K$-group $K^0(UConf_n(G))$. Following this, we ask whether the above theorem can be extended in some way to multiplicative \emph{cohomology} theories with Noetherian coefficient ring $E^*$. That is, whether the functors 
\[
\Gamma \mapsto E^i(UConf_n(\Gamma))
\]
are in some sense finitely generated representations of the graph-minor category.

An approach to finite generation of cohomology is to consider its linear duals. Explicitly, one could try and apply the Noetherian property of the opposite graph-minor category to prove finite generation for the functors $\operatorname{Hom}_{E_*}(E^i(UConf_n(-)), E_0)$ or related constructions. This idea appears e.g. in \cite{kupers2018representation}, where Kupers and Miller prove that duals of homotopy groups of configuration spaces are finitely generated FI-modules. Thus we propose,

\begin{conjecture}
For every fixed $n$, the dual of the Grothendieck group of vector bundles over configuration spaces of graphs
\[
\Gamma \mapsto \operatorname{Hom}_{\mathbb{Z}}(K^0(UConf_n(\Gamma)),\mathbb{Z})
\]
is generated by finitely many functions on vector bundles under deletion and contraction.
\end{conjecture}


\end{remark}

\subsection{Acknowledgements}
We are deeply grateful to AIM for facilitating the workshop on Configuration Spaces of Graphs, Feb 2020, at which this project emerged. We also thank John Wiltshire-Gordon for suggesting this problem, and Safia Chettih, John Wiltshire-Gordon, and Ben Knudsen for helping us develop the ideas presented here. Special thanks to Gabriel Drummond-Cole for being a key part of this project throughout the workshop and for providing various important suggestions and advice.

\section{Key Lemma}
In this section we present and prove a lemma that will be central to the geometric constructions of this paper. Let $U\subseteq \Gamma$ be an open set and denote
$Conf_{n,k}(\Gamma,U)$
for the subspace of configurations of $n$ points in $\Gamma$ (either ordered or unordered) no more than $k$ of which lie in $U$.
\begin{lemma} \label{lem:one_point}
Fix a vertex $v\in \Gamma$ and let $U$ be the ball of radius $1/2$ around $v$. Then the inclusion
\[
Conf_{n,1}(\Gamma,U)\into Conf_n(\Gamma)
\]
is a homotopy equivalence.
\end{lemma}
\begin{proof} 
For $r>0$ let $f_r: [0,1]\to [0,1]$ be the monotonic homeomorphism $x\mapsto x^r$. This defines a homeomorphism $h_r$ of the graph $\Gamma$ treated as a CW-complex, so that every edge comes equipped with an identification with $[0,1]$ as follows: map every edge to itself via the identity map, except for edges incident on $v$ -- orient these edges so that $v$ is identified with $0\in [0,1]$ and map them to themselves via $f_r$.

As the parameter $r$ varies, the functions $h_r$ assemble to a continuous isotopy $h:(0,1]\times \Gamma \to \Gamma$, and therefore we have an induced isotopy 
\[
H: (0,1]\times Conf_n(\Gamma)\to Conf_n(\Gamma)
\]
on the configuration spaces (either ordered or unordered), where at map at time $r$ is denoted by $H_r$.

Now let $d_2: Conf_n(\Gamma)\to [0,\infty]$ be distance of the 2nd closest point to $v$. To be clear, if there are two equidistant points closest to $v$ then $d_2$ will take on this minimal distance, and if there is no more than one point in the connected component of $v$ then $d_2$ will take the value $\infty$. Clearly this is a continuous function, similarly to how $\min(x,y)$ is continuous on $\R^2$.

With these at hand, consider the continuous map
\[
H_{d_2\wedge 1}: Conf_n(\Gamma) \to Conf_n(\Gamma) \,,\quad \bar{x}\mapsto H_{d_2(\bar{x})\wedge 1}(\bar{x}).
\]
First, observe that this map is well defined: $d_2>0$, as there can not be two distinct points at distance $0$ to $v$. The effect of this map on configurations is to push points away from $v$: if a point $x_i\in \Gamma$ of the configuration is at distance $0<s\leq 1$ from $v$, then it is mapped to the point on the same edge but at distance $s^{d_2\wedge 1}\geq s$ from $v$.

After application of the above map, the second closest point to $v$ will be at distance $d_2(\bar{x})^{d_2(\bar{x})\wedge 1}$. But this function always takes value greater than $1/2$ (recall that the minimum of $x^x$ is $e^{-1/e}\geq 0.6$). Thus there is at most one point at distance $\leq 1/2$ to $v$.

Lastly, the isotopy $H_{(1-t)+t(d_2\wedge 1)}$ connects the identity on $Conf_n(\Gamma)$ at $H_1$ with the above map that pushes all but the closest point away from $v$.
\end{proof}

\section{Contraction}
We wish to realize geometrically the An--Drummond-Cole--Knudsen homological edge contraction map from \cite[Appendix C]{an2017subdivisional}. For this purpose we use the following,
\begin{lemma}
If $U_i\subseteq \Gamma_i$ are open sets such that $\Gamma_1\setminus U_1 = \Gamma_2\setminus U_2$ and $\bar{U}_1\simeq \bar{U}_2$ are homotopy equivalent relative to $\partial \bar{U}_1=\partial \bar{U}_2$, then \[Conf_{n,1}(\Gamma_1,U_1) \simeq Conf_{n,1}(\Gamma_2,U_2).\] 
\end{lemma}
\begin{proof}
For any continuous function $f:\bar{U}_i\to \bar{U}_j$ that fixes the boundary, one gets a map on configuration spaces $M(f):Conf_{n,1}(\Gamma_i,U_i) \to Conf_{n,1}(\Gamma_j,U_j)$ by the rule that every point not in $U_i$ is mapped to itself, and the (at most one) point in $U_i$ is mapped to $U_j$ via $f$. Note that this definition patches to a continuous map.

Now for $\{i,j\}=\{1,2\}$ let $f_{ij}: \bar{U}_i\to \bar{U}_j$ be maps such that $f_{ji}\circ f_{ij}$ is homotopic rel $\partial$ to $Id_{U_i}$, say via the homotopy $h^{(i)}_t$. Then the induced maps on configurations $M(f_{ij})$ are homotopy equivalences, where the homotopies are given by $M(h^{(i)}_t)$.
\end{proof}

\begin{corollary}\label{cor:contraction}
The ``contraction of a subtree" map on homology is realized geometrically as follows.

Let $T\subseteq \Gamma$ be a tree and let $U\supseteq T$ be an $\epsilon$-neighborhood with $\epsilon\leq 1/2$. Then the contraction is realized by the span
\[
Conf_n(\Gamma/T)\, \tilde{\leftarrow}\, Conf_{n,1}\,(\Gamma/T,U/T)\, \tilde{\to}\, Conf_{n,1}(\Gamma,U) \into Conf_n(\Gamma)
\]
\end{corollary}
Now, the cone of the contraction can be described explicitly: it is the space of configurations that have at least two points in $U$, where all other configurations are collapsed to a point.

\section{Deletion} \label{sec:deletion}
Let $\Gamma=(V,E)$ be a graph. A \emph{half-edge} $h$ in $\Gamma$ is formally an incident pair $(v,e)\in V\times E$. Denote $e(h):=e$ and $v(h):=v$. Geometrically, consider the half-edge $h$ to be the subspace of the edge $e(h)$ identified with the interval $(0,1/2)\subset [0,1]$, parametrized so that $0$ corresponds to the vertex $v(h)$. We will abuse the notation and freely treat $h$ as an open set in the topological space $\Gamma$.

\begin{theorem}
Let $H=\{h_1,\ldots,h_k\}$ be a subset of the half-edges incident to $v$. Then the cone of the inclusion $\iota_H: Conf_n(\Gamma\setminus \cup H) \into Conf_n(\Gamma)$ is homotopy equivalent to the reduced suspensions
\[
\bigvee_{h\in H} \tilde{\Sigma}Conf_{n-1}(\Gamma\setminus \{v\})_+.
\]

In the ordered case the cone is equivalent to the induction of this wedge with its $S_{n-1}$-action to $S_n$. That is,
$S_n\wedge_{S_{n-1}}\bigvee_{h\in H} \tilde{\Sigma}(Conf_{n-1}(\Gamma\setminus \{v\})_+.
$
\end{theorem}

\begin{proof}
We start with the unordered case. Let $U$ be the ball of radius $1/2$ around $v$. Lemma \ref{lem:one_point} gives a homotopy equivalence of pairs 
\[
(Conf_{n,1}(\Gamma, U), Conf_{n,1}(\Gamma\setminus \cup H, U\setminus \cup H)) \simeq (Conf_{n}(\Gamma), Conf_n(\Gamma\setminus \cup H))
\]
and the former is a cofibration. Thus, since cones are homotopy invariant, it is sufficient to construct a homeomorphism
\[
Conf_{n,1}(\Gamma, U)/Conf_{n,1}(\Gamma\setminus \cup H, U\setminus \cup H) \to \bigvee_{h\in H} \tilde{\Sigma}Conf_{n-1}(\Gamma\setminus \{v\})_+.
\]

Suppose a configuration i $Conf_{n,1}(\Gamma, U)$ has a point $x$ on $h\in H$. Then map it to the suspension labelled by $h$, where the configuration is obtained by forgetting $x$ and setting the suspension parameter equal to $2d(x,v)$. If no such $x$ exists, map the configuration to the basepoint $*$. Note that this function is well-defined, as there can be at most one point on our set of half-edges. Note also that if $x=v$ then the suspension parameter is set to $0$ which lands on the basepoint. Lastly, as $x$ leaves the half-edge, the cone parameter goes to $1$ and its image under our map will approach the basepoint.

Since configurations in $\Gamma\setminus \cup H$ are all sent to the basepoint, the map above factors through the quotient by $Conf_{n,1}(\Gamma\setminus \cup H, U\setminus \cup H)$.

The inverse map is defined by sending a point $(\bar{x},t)\in \tilde{\Sigma} Conf_{n-1}(\Gamma\setminus\{v\})_+$ on the wedge-summand labeled by $h$ to the configuration that has an additional point on $h$ at distance $t/2$ to $v$. Of course, the basepoint has to map to the basepoint. These maps clearly patch to a continuous map, and are inverses to the above maps that would forget the newly added point.

To adapt the above argument for the ordered case, index the wedge sum by $[n]\times H$, accounting for the label of the point $x_i\in h$. The rest of the construction works in just the same way. Clearly this construction is compatible with the $S_n$-action, and the stabilizer of $n$ acts by the ordinary permutation action on $Conf_{n-1}(\Gamma\setminus \{v\})$.
\end{proof}
Lastly, we wish to describe the "boundary map" $Cone(\iota_H) \to \tilde{\Sigma}Conf_n(\Gamma\setminus \cup H)_+$ obtained by crushing the base of the cone $\cong Conf_n(\Gamma)$ to a point.
\begin{proposition}
Under the identification \[Cone(\iota_H)\simeq \bigvee_{h\in H}\tilde{\Sigma} Conf_{n-1}(\Gamma\setminus \{v\})_+\]
the boundary map to $\tilde{\Sigma} Conf_{n}(\Gamma\setminus \cup H)_+$ on the wedge summand labeled by $h$ has the form
\[
\tilde\Sigma add_{v} - \tilde\Sigma add_{h}
\]
where the map $add_v$ adds the vertex $v$ to a configuration, while $add_h$ adds a new closest point to $v$ on the edge $e(h)$.
\end{proposition}
\begin{proof}
An explicit homotopy equivalence from the collapse
\[
Conf_{n,1}(\Gamma, U)/Conf_{n,1}(\Gamma\setminus \cup H, U\setminus \cup H)
\]
to the mapping cone on the inclusion can be constructed by the following recipe. For a configuration with a point $x_i$ on a half-edge $h\in H$, say at distance $d=d(x_i,v)$,
\begin{itemize}
    \item when $d\leq 1/6$, i.e. $x_i$ is on the third of $h$ nearest to $v$, move $x_i$ to $v$ while fixing all other points and set the cone parameter to $6d$;
    \item when $1/2-d\leq 1/6$, i.e. $x_i$ is on the third of $h$ farthest away from $v$, scale the edge $e(h)$ down away from $v$, moving all points on it until $x_i$ appears at distance $1/2$ and set the cone parameter to $6(1/2-d)$;
    \item lastly, when $x_i$ is in the middle third of $h$, i.e. $1/6\leq d \leq 2/6$, scale this middle third up to encompass the whole of $h$. This amounts to moving $x_i$ to be at distance $3(d-1/4) + 1/4$ to $v$. All other points remain fixed. 
\end{itemize}
These three maps patch together to give the desired homotopy equivalence (e.g. a homotopy $H_t$ can be constructed by replacing the terms $1/6$ in the above construction with $t/6$ and adjusting the formulas accordingly).

Now, the boundary map is constructed by collapsing the base of the mapping cone to a point. This is the set of points in $Conf_{n,1}(\Gamma, U)$, which in the above construction is the range of configurations with a point in the middle third of some $h\in H$.
But under the homeomorphism with $\bigvee_{h\in H}\tilde{\Sigma} Conf_{n-1}(\Gamma\setminus \{v\})_+$, such configurations correspond to the middle thirds in every suspension. After collapsing the base of the cone to a point, the boundary map factors through maps of the form $\Sigma X \to \Sigma_{bottom} X \vee \Sigma_{top} X$ where one collapses the middle third to a point. Now observe that on the bottom third of the suspensions, the boundary map is precisely $\Sigma add_v$, while on the top third we find maps that add a new closest point to $v$ on the edge $e(h)$ but with the suspension parameter going backwards (this is the $-d$ term appearing in the cone parameter). Thus this map is homotopic to $-\Sigma add_{h}$.

The same arguments apply in the case of ordered configurations.
\end{proof}

\section{Applications: Generalized homology theories} \label{sec:applications}

By the previous section, the Puppe cofiber sequence for the inclusion of half-edge deletion takes the form
\begin{eqnarray*}
 UConf_n(\Gamma\setminus \cup H) &\into& UConf_n(\Gamma) \to
\bigvee_{h\in H} \tilde{\Sigma}UConf_{n-1}(\Gamma\setminus \{v\})_+ \to \\
\to \tilde\Sigma UConf_n(\Gamma\setminus \cup H)_+ &\into&  \tilde\Sigma UConf_n(\Gamma)_+ \to \ldots
\end{eqnarray*}
where the connecting map between the rows is given on the wedge-summand with label $h\in H$ by $\tilde\Sigma add_v- \tilde\Sigma add_h$. Also, as mentioned in the previous section, a similar sequence exists in the $S_n$-equivariant context for the ordered configurations spaces, but with an appropriate induction on configurations of $n-1$ points. 

Applying any generalized homology or cohomology theory to this Puppe sequence now yields the long exact sequence claimed in Theorem \ref{thm:LES}.

We next turn to the proof of Theorem \ref{thm:fin-gen-homology}, regarding the generalized homology of unordered configuration spaces as a representation of the Miyata-Proudfoot-Ramos opposite graph-minor category defined in \cite{miyata2020categorical}.


\begin{proof}[Proof of Theorem \ref{thm:fin-gen-homology}]
The basic inputs to the proof are the Noetherianity of the coefficient ring $E_{\ast}$, the fact that the singular chains $C_{i}(UConf_{n}(-))$ are equivalent to a finitely generated representation of the opposite graph minor category, and the Atiyah-Hirzebruch spectral sequence -- the AHSS for short.

Atiyah-Hirzebruch provide a spectral sequence for every space $X$ with first page
\[
E^1_{p,q} = C_{p}(X;E_{q}) := C_{p}(X) \otimes E_{q} \implies E_{p+q}(X).
\]
More explicitly, $E^{1}$-page of the spectral sequence looks as follows-            \begin{figure}[H]
\centering
    \begin{tikzpicture}[scale = 0.85, every node/.style={transform shape}]
    \draw (-0.5,-0.2) -- (12,-0.2);
    \draw (-0.2,-0.5) -- (-0.2,3);
    \node at (1,0) {$C_{0}(X) \otimes E_{0}$};
    \node at (4,0) {$C_{1}(X) \otimes E_{0}$};
    \node at (7,0) {$C_{2}(X) \otimes E_{0}$};
    \node at (10,0) {$C_{3}(X) \otimes E_{0}$};
    \node at (1,1) {$C_{0}(X) \otimes E_{1}$};
    \node at (4,1) {$C_{1}(X) \otimes E_{1}$};
    \node at (7,1) {$C_{2}(X) \otimes E_{1}$};
    \node at (10,1) {$C_{3}(X) \otimes E_{1}$};
    \node at (1,2) {$C_{0}(X) \otimes E_{2}$};
    \node at (4,2) {$C_{1}(X) \otimes E_{2}$};
    \node at (7,2) {$C_{2}(X) \otimes E_{2}$};
    \node at (10,2) {$C_{3}(X) \otimes E_{2}$};

    \node at (1,-0.5) {0};
    \node at (4,-0.5) {1};
    \node at (7,-0.5) {2};
    \node at (10,-0.5) {3};
    \node at (-.5,0) {0};
    \node at (-.5,1) {1};
    \node at (-.5,2) {2};
    

    \draw[->] (0,0) -> (-1,0); 
    \draw[->] (3,0) -> (2,0);
    \draw[->] (6,0) -> (5,0);
    \draw[->] (9,0) -> (8,0); 
    \draw[->] (0,1) -> (-1,1); 
    \draw[->] (3,1) -> (2,1);
    \draw[->] (6,1) -> (5,1);
    \draw[->] (9,1) -> (8,1);
    \node at (12,1) {$\cdots$};
    \node at (4,2.8) {$\cdot$};
    \node at (4,2.9) {$\cdot$};
    \node at (4,3) {$\cdot$};
    \node at (7,2.8) {$\cdot$};
    \node at (7,2.9) {$\cdot$};
    \node at (7,3) {$\cdot$};
    \end{tikzpicture}
\end{figure}
\cite{an2017subdivisional} shows that the singular chains $C_*(UConf_n(\Gamma))$ are quasi-isomorphic to a smaller chain complex, the reduced \'{S}wi\k{a}tkowski complex $\widetilde{S}_{*,n}(\Gamma)$, for every graph $\Gamma$. We may thus replace replace every occurrence of $C_*(X)$ in the AHSS above with these \'{S}wi\k{a}tkowski complexes. 

\cite{an2017subdivisional} further shows that $\widetilde{S}_{*,n}(-)$ is a representation of the opposite graph-minor category. But our geometric construction of edge contraction in Corollary \ref{cor:contraction} shows that the quasi-isomorphism with singular chains is compatible with a homotopy action of the opposite graph-minor category. Furthermore \cite[Proof of 1.15]{miyata2020categorical} shows that $\widetilde{S}_{i,n}(-)$ is finitely generated as a representation of the opposite graph-minor category for every $i\geq 0$.

Note that all terms $\widetilde{S}_{i,n}(-) \otimes E_{j}$ are $E_{0}$-modules, and recall the fact that $E_{\ast}$ being Noetherian implies $E_{0}$ is a Noetherian ring and every $E_{j}$ is a finitely generated $E_{0}$-module. Thus all terms in the $E^1$-page above, with $C_*$ replaced by $\widetilde{S}_{*,n}$, form finitely generated representations of the opposite graph minor category over the ring $E_{0}$.

Now, all terms in the later pages of the spectral sequence are subquotients of those appearing in the first page. But \cite[Theorem 1.2]{miyata2020categorical} shows that finitely generated representations of the opposite graph-minor category over a Noetherian ring are again Noetherian, and thus finite generation passes to subquotients. The same can be said of the terms at the $E^{\infty}$-page, which gives us the graded factors of a filtration on the functor $E_{i}(UConf_{n}(-))$. Thus the latter functor is filtered with quotients finitely generated as representations of the opposite graph minor category. Hence, it is itself a finitely generated representation.
\end{proof}

\bibliographystyle{alpha}
\bibliography{refs}

\end{document}